\documentclass[reqno]{amsart}
\usepackage{amsmath, amssymb, amsthm}
\usepackage{url}
\usepackage{moreverb}
\usepackage{algorithm}
\usepackage{algorithmicx}
\usepackage{algpseudocode}
\usepackage{pgfplots}
\usepackage{alltt}
\usepackage{listings}
\usepackage{xspace}

\theoremstyle{plain}
\newtheorem{prop}{Proposition}[section]
\theoremstyle{definition}
\newtheorem{defi}[prop]{Definition}

\newtheorem{exam}[prop]{Example}
\newtheorem{lem}[prop]{Lemma}

\newcommand{\gr}[1]{{\color{gray} #1}}
\newcommand{\ie}{\emph{i.e.}}
\newcommand{\Cilk}{\texttt{Cilk}\xspace}
\newcommand{\CilkP}{\texttt{Cilk++}\xspace}
\newcommand{\CPP}{\texttt{C++}\xspace}

\renewcommand{\leq}{\leqslant}
\renewcommand{\geq}{\geqslant}
\newcommand{\MMX}{\texttt{MMX}\xspace}
\newcommand{\SIMD}{\texttt{SIMD}\xspace}
\newcommand{\SSE}{\texttt{SSE}\xspace}
\newcommand{\SSEV}{\texttt{SSE4.1}\xspace}
\newcommand{\NN}{\mathbb{N}}
\renewcommand{\tt}[1]{\texttt{#1}}
\newcommand{\sgnode}[1]{{\bf \left<#1\right>}}
\renewcommand{\ie}{\emph{i.e}}
\DeclareMathOperator{\Irr}{Irr}

\DeclareMathOperator{\card}{card}

\newcommand{\Delete}[1]{}

\title{Exploring the tree of numerical semigroups}
\author{Jean Fromentin and  Florent Hivert}
\date{April 10,  2014}
\keywords{semigroups, tree, algorithm, \texttt{SSE}, 
\texttt{Cilk}, optimization}

\begin{document}

\maketitle

\begin{abstract}
  In this paper we describe an algorithm visiting all numerical semigroups
  up to a given genus using a well suited representation. The
  interest of this algorithm is that it fits particularly well the
  architecture of modern computers allowing very large optimizations: we obtain the number of numerical semigroups of genus $g\leq  67$ and we confirm the Wilf conjecture for $g\leq 60$.
  
\end{abstract}

\vspace{1em}

\section*{Introduction}

A \emph{numerical semigroup} $S$ is a subset of $\NN$ containing $0$, closed under addition and of finite complement in $\NN$.  
For example the set 
\begin{equation}
\label{E:NSG}
S_E=\{0,3,6,7,9,10\}\cup\{x\in\NN, x\geq 12\}
\end{equation}
is a numerical semigroup.
The \emph{genus} of a numerical semigroup $S$, denoted by~$g(S)$, is the cardinality of~$\NN\setminus S$.
 For example the genus of $S_E$ is $6$,  the cardinality of $\{1,2,4,5,8,11\}$.

For a given positive integer $g$, the number of numerical semigroups of genus $g$ is finite and is denoted by $n_g$. 
In  J.A. Sloane's \emph{on-line encyclopedia of integer sequences}~\cite{OEIS} we find the values of $n_g$ for $g\leq 52$. 
These values have been obtained by M. Bras-Amor\'os (\cite{BrasAmoros2008} for more details for $g\leq 50)$. 
On his home page~\cite{Delgado}, M. Delgado  gives the value of $n_{55}$.

M.~Bras-Amor\'os used a depth first search exploration of the tree of numerical semigroups $\mathcal{T}$ up to a given genus.
This tree was introduced by J.C.~Rosales and al. in \cite{Rosales} and it is the subject of Section~\ref{S:Tree}.
Starting with all the numerical semigroups of genus $49$ she obtained the number of numerical semigroups of genus $50$ in $18$ days on a pentium D runing at $3$GHz. 
In the package~\tt{NumericalSgs} \cite{NumericalSgps} of \tt{GAP} \cite{GAP}, M.~Delgado together with P.A.~Garcia-Sanchez and J. Morais used the same method of exploration.

Here we describe a new algorithm for the exploration of the tree of numerical
semigroups $\mathcal{T}$ and achieve the computation of $n_g$ for $g\leq 67$.
The cornerstone of our method is a combinatorial representation of
numerical semigroups that is well suited and allows large code optimization essentially based on the use of
vectorial instructions and parallelization.
The goal of the paper is twofold:
first to present our encoding of numerical semigroups and the associated
algorithms, and second to present the optimization techniques which allow, for
those kinds of algorithms, to get speedups by factors of hundreds and even
thousands. We claim that these techniques are fairly general for those kinds of
algorithms. As a support for the claim, we applied it to an algorithm of
N.~Borie enumerating integer vector modulo permutation groups~\cite{Borie} and
got a speedup by a factor larger than 2000 using 8 cores.

The paper is divided as follows.
In Section~\ref{S:Tree} we describe the tree of numerical semigroups and give bounds for some parameters attached to a numerical semigroup.
The description of our representation of numerical semigroups is done in the second section.
In Section~\ref{S:Algo} we describe an algorithm based on the representation given in Section~\ref{S:DecNumber} and give its complexity. 
Section~\ref{S:Opti} is more technical, and is devoted to the optimization of the algorithm introduced in Section~\ref{S:Algo}.
In the last section we emphasize the results obtained using our algorithm.

\section{The tree of numerical semigroups}
\label{S:Tree}

We start this section with definitions and properties of numerical semigroups that will be used in the sequel.
For a more complete introduction, the reader can usefully consult the book \emph{Numerical Semigroups} by J.C.~Rosales and P.A.~Garc\'ia-S\'anchez \cite{BookNS} or the book \emph{The Diophantine Frobenius Problem} by J.L.~Ram\'irez~Alfons\'in \cite{BookDFP}.

\begin{defi}
Let $S$ be a numerical semigroup. We define 

$i)$ $m(S)=\min(S\setminus\{0\})$, the \emph{multiplicity} of $S$;

$ii)$ $g(S)=\card(\NN\setminus S)$, the \emph{genus} of $S$;

$iii)$ $f(S)=\max(\mathbb{Z}\setminus S)$, the \emph{Frobenius} of $S$;

$iv)$ $c(S)=f(S)+1$, the \emph{conductor of $S$}. 
\end{defi}

By definition a numerical semigroup is an infinite object and we need a finite description of such an object. 
That is provided by generating sets.

\begin{defi}
A subset $X=\{x_1<x_2<...<x_n\}$ of a semigroup is a \emph{generating set} of $S$ if every element of $S$ can be expressed as a sum of elements in $X$. 
In this case we write $S=\left<x_1,...,x_n\right>$.
\end{defi}

If we reconsider the numerical semigroup of \eqref{E:NSG}, we obtain
\begin{equation}
\label{E:GNSG}
S_E=\{0,3,6,7,9,10\}\cup[12,+\infty[=\left<3,7\right>.
\end{equation}

A non-zero element $x$ of a numerical semigroup $S$ is said to be \emph{irreducible} if it cannot be expressed as a sum of two non-zero elements of $S$.  
We denote by $\Irr(S)$ the set of all irreducible elements of $S$.

\begin{lem}[Lemma 2.3 of~\cite{BookNS}]
For a numerical semigroup $S$, the set $\Irr(S)$ is the minimal generating set of $S$ relatively to the inclusion ordering.
\end{lem}

The different parameters we have defined on a numerical semigroup, satisfy the following relations.

\begin{prop}[Proposition~2.12 and Lemma~2.14 of~\cite{BookNS}]
\label{P:Res}
For every  numerical semigroup  $S$,  we have

$i)$  $x\in \Irr(S)$ implies $x\leq c(S)+m(S)-1$;

$ii)$ $m(S)\leq g(S)+1$;

$iii)$ $c(S)\leq 2 g(S)$.
\end{prop}

A consequence of Proposition~ \ref{P:Res}~$i)$ is that $\Irr(S)$ is finite and its cardinality is at most $c(S)+m(S)-1$.
Moreover, the cardinality of $\Irr(S)$ is at most $m(S)$ since any two distinct elements of $\Irr(S)$ cannot be congruent modulo $m(S)$. See Section~2 of Chapter I of~\cite{BookNS} for more details.

We now explain the construction of the tree of numerical semigroups.
Let~$S$ be a numerical semigroup. The set $S'=S\cup\{f(S)\}$ is also a numerical  semigroup and its genus is $g(S)-1$. 
As  each integer greater than $f(S)$ is included in $S'$ we have $c(S')\leq f(S)$. 
Therefore every semigroup $S$ of genus $g$ can be obtained from a semigroup $S'$ of genus $g-1$ by removing an element of~$S'$ greater than or equal to $c(S')$.

\begin{prop}[Proposition~7.28 of \cite{BookNS}]
\label{P:Sx}
Let $S$ be a numerical semigroup and $x$ an element of $S$. The set~$S^x=S\setminus\{x\}$ is a numerical semigroup if and only if $x$ is irreducible in $S$.
\end{prop}

Proposition~\ref{P:Sx} implies that every semigroup $S$ of genus $g$ can be obtained from a semigroup $S$ by removing a generator $x$ of $S$ that is greater than or equal to $c(S)$.
Hence the relation~$S'=S^x$ holds.

We construct the tree of numerical semigroups, denoted by $\mathcal{T}$ as follows. 
The root of the tree is the unique semigroup of genus $0$, \ie, $\left<1\right>$ that is equal to $\NN$. 
If~$S$ is a semigroup in the tree,  the sons of $S$ are exactly the semigroups~$S^x$ where $x$ belongs to $\Irr(S)\cap[c(S),+\infty[$. 
By convention, when depicting the tree, the numerical semigroup $S^x$ is in the left of $S^y$ if $x$ is smaller than~$y$. 
With this construction, a semigroup $S$ has depth $g$ in $\mathcal{T}$ if and only if its genus is~$g$, see Figure~\ref{F:Tree}.
We denote by $\mathcal{T}_{g}$ the subtree of~$\mathcal{T}$ restricted to all semigroups of genus $\leq g$.

 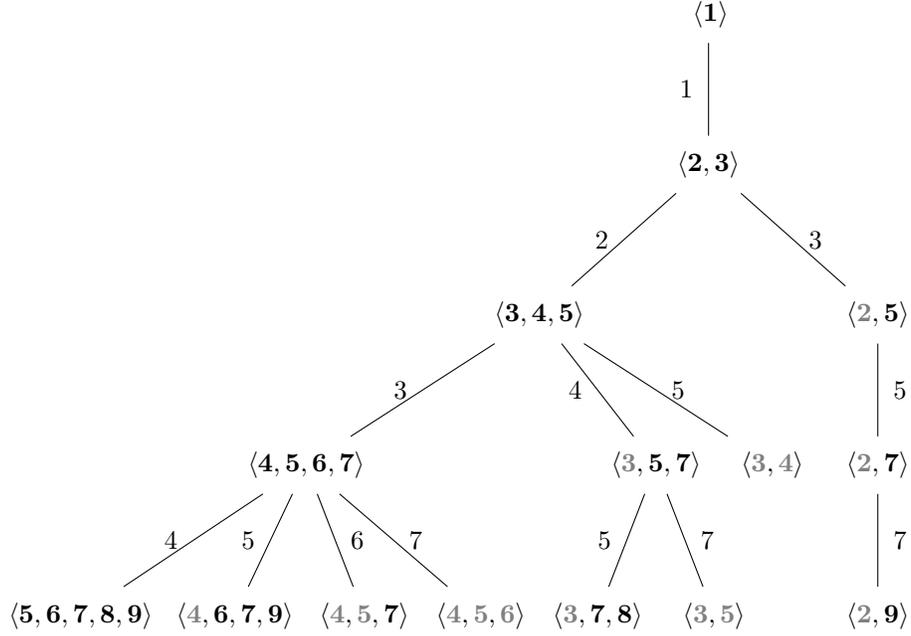
\begin{figure}[h!]
\[
\begin{tikzpicture}[level distance=2cm, inner sep=2mm,level/.style={sibling distance=1.55cm}]
    \node {$\sgnode{1}$}
    child {node {$\sgnode{2,3}$}
      child [sibling distance=4.5cm] {node {$\sgnode{3,4,5}$}
        child {node {$\sgnode{4,5,6,7}$}
          child [sibling distance=2cm]{node {$\sgnode{5,6,7,8,9}$}  edge from parent node [left] {4}}
          child [sibling distance=1.9cm]{node {$\sgnode{\gr 4,6,7,9}$}  edge from parent node [left] {5}}
          child {node {$\sgnode{\gr 4,\gr 5,7}$}  edge from parent node [right] {6}}
          child {node {$\sgnode{\gr 4,\gr 5,\gr 6}$}  edge from parent node [right] {7}}
          edge from parent node [left] {3}
        }
        child{{} edge from parent[draw=none]}
        child{{} edge from parent[draw=none]}
        child {node {$\sgnode{\gr 3,5,7}$}
          child {node {$\sgnode{\gr 3,7,8}$} edge from parent node [left] {5}}
          child {node {$\sgnode{\gr 3,\gr 5}$} edge from parent node [right] {7}}
          edge from parent node [left] {4}
        }
        child {node {$\sgnode{\gr 3,\gr 4}$}
          edge from parent node [right] {5}
        }
        edge from parent node [left] {2}
      }
      child [sibling distance=4.5cm] {node {$\sgnode{\gr2,5}$}
        child {node {$\sgnode{\gr2,7}$}
          child {node {$\sgnode{\gr2,9}$}
            edge from parent node [right] {7}
          }
          edge from parent node [right] {5}
        }
        edge from parent node [right] {3}
      }
      edge from parent node [left] {1}
    }
    ;
  \end{tikzpicture}
  \]
\caption{The first four layers of the tree $\mathcal{T}$ of numerical semigroups, corresponding to~$\mathcal{T}_4$. A generator of a semigroup is it in gray if is not greater than $c(S)$. An edge between a semigroup $S$ and its son $S'$ is labelled by  $x$ if $S'$ is obtained from $S$ by removing~$x$, that is if $S'=S^x$ holds.}
\label{F:Tree}
\end{figure}

\section{Decomposition number}
\label{S:DecNumber}

The aim of this section is to describe a representation of numerical semigroups,
which is well suited to an efficient exploration of the tree~$\mathcal{T}$ of numerical semigroups. 

\begin{defi}
Let $S$ be a numerical semigroup.
For every $x$ of $\NN$  we set
\[
D_S(x)=\{y \in S\ |\ x-y\in S\ \text{and}\ 2y\leq x\}
\]
and $d_S(x)=\card D_S(x)$.
We called $d_S(x)$ the $S$-\emph{decomposition number} of~$x$.
The application $d_S:\NN\to\NN$ is the $S$-\emph{decomposition numbers function}.
\end{defi}

Assume that $y$ is an element of $D_S(x)$.
By definition of $D_S(x)$, the integer~$z=x-y$ also belongs to $S$. 
Then $x$ can be decomposed as $x=y+z$ with $y$ and $z$ in~$S$. 
Moreover the condition $2y\leq x$ implies $y\leq z$. 
In other words if we define~$D'_S(x)$ to be the set of all $(y,z)\in S\times S$ with $x=y+z$ and $y\leq z$ then~$D_S(x)$ is the image of $D'_S(x)$ under the projection on the first coordinate.
 Hence~$D_S(x)$ describes how~$x$ can be decomposed as sums of two elements of $S$.
 This justifies the name given to the function $d_S$.
 
\begin{exam}
Reconsider the semigroup $S_E$ given at \eqref{E:NSG}.
The integer $14$ admits two decompositions as sums of two elements of $S$, namely $14=0+14$ and $14=7+7$. 
Thus the set $D_{S_E}(14)$ is equal to $\{0,7\}$ and $d_{S_E}=2$ holds.
\end{exam}

\begin{lem}
\label{L:RepN}
For every numerical semigroup $S$ and every integer $x\in \NN$,  we have $d_S(x)\leq 1+\left\lfloor \dfrac{x}2\right\rfloor$ and the equality holds for $S=\NN$.
\end{lem}

\begin{proof}
As the set $D_S(x)$ is included in $\left\{0,...,\left\lfloor\frac x2\right\rfloor\right\}$, the relation $d_S(x)\leq 1+\left\lfloor \frac{x}2\right\rfloor$ holds.
For $S=\NN$ we have the equality for the set~$D_S(x)$ and so for the integer $d_S(x)$.
\end{proof}

A straightforward consequence of the definition of $S$-decomposition numbers is :

\begin{prop}
\label{P:Dandd}
For a numerical semigroup $S$ and $x\in\NN\setminus\{0\}$, we have:

 $i)$ $x$ lies in $S$ if and only if $d_S(x)>0$.

 $ii)$ $x$ is in $\Irr(S)$ if and only if $d_S(x)=1$.
\end{prop}

We note  that $0$ is never irreducible despite the fact $d_S(0)$ is $1$ for all numerical semigroups $S$.
We now explain how to compute the $S$-decomposition numbers function of a numerical semigroup  from that of its father.

\begin{prop}
Let $S$ be a numerical semigroup and $x$ be an irreducible element of $S$. 
Then for all~$y\in \NN\setminus\{0\}$ we have
\[
d_{S^x}(y)=\begin{cases}
d_S(y)-1&\text{if $y\geq x$ and $d_S(y-x)>0$,}\\
d_S(y) &\text{otherwise.}
\end{cases}
\] 
\end{prop}

\begin{proof}
A direct consequence of~$D_{S^x}(y)=D_S(y)\setminus\{y-x,x\}$.
\end{proof}

\section{A new algorithm}
\label{S:Algo}
We can easily explore the tree of numerical semigroups up to a genus~$G$ using
a depth first search algorithm (see Algorithm~\ref{A:ExpRec}). This approach does not seem to
have been used before.  In particular, M.~Bras-Amor\'os and M.~Delgado use
instead a breadth first search exploration.  The main advantage in our
approach is the small memory needs.  
Indeed, in the case of breadth first search algorithm one needs to compute and store the list $L_g$ of all numerical semigroups of genus $g$ before 
visiting the numerical semigroups of genus $g+1$, which is not required for a depth first search exploration.
In this paper we are interested in the exploration of the list~$L_g$ for $g\in \NN$, not in storing it.
This is the reason we use a depth first search algorithm for the exploration of the tree $\mathcal{T}_g$.
The only limitation is then the duration of the exploration and not the amount of available memory.
For example the list~$L_{54}$ needs several terabytes to be stored. 

\begin{algorithm}
\small
\caption{\small Recursive Depth first search exploration of the tree $\mathcal{T}_g$.}\label{A:ExpRec}
\begin{algorithmic}[1]
\Procedure{ExploreRec}{\tt{S}, \tt{G}}
	\If{$g(\tt{S})< \tt{G}$}
		\For{\tt{x} from $c(\tt{S})$ to $c(\tt{S})+m(\tt{S})$}
			\If{$\tt{x}\in \Irr(\tt{S})$}
				\State \Call{ExploreRec}{$\tt{S}^\tt{x}, \tt{G}$}
			\EndIf
		\EndFor
	\EndIf
\EndProcedure
\end{algorithmic}
\end{algorithm}

\noindent Equivalently, we can use an iterative version which uses a stack:

\begin{algorithm}
\caption{\small Iterative Depth first search exploration of the tree $\mathcal{T}_g$}\label{A:Exp}
\small
\begin{algorithmic}[1]
\Procedure{Explore}{\tt{G}}
\State Stack \tt{stack} \Comment{the empty stack}
\State \tt{stack.push($\mathbb{N}$)}
\While{\tt{stack} is not empty} 
	\State $\tt{S} \gets \tt{stack.top()}$
	\State \tt{stack.pop()}
	\If{$g(\tt{S})< \tt{G}$}
		\For{\tt{x} from $c(\tt{S})$ to $c(\tt{S})+m(\tt{S})$}
			\If{$\tt{x}\in \Irr(\tt{S})$}
				\State $\tt{stack.push}(\tt{S}^\tt{x})$
			\EndIf
		\EndFor
	\EndIf
\EndWhile
\EndProcedure
\end{algorithmic}
\end{algorithm}

In Algorithm~\ref{A:ExpRec} we do not specify how to compute $c(S)$, $g(S)$ and $m(S)$  from $S$ neither how to test if an integer is irreducible. 
It also misses  the characterization of $S^x$ from $S$.
These items depend heavily of the representation of $S$.
Our choice is to use the $S$-decomposition numbers function. 
The first task is to use a finite set of such numbers to characterize the whole semigroup.

\begin{prop}
\label{P:delta}
Let $G$ be an integer and $S$ be a numerical semigroup of genus $0<g\leq G$.
Then $S$ is entirely described by the vector $\delta_S=(d_S(0),...,d_S(3\,G)) \in \NN^{3G+1}$. More precisely we can obtain $c(S), g(S), m(S)$ and $\Irr(S)$ from $\delta_S$.
\end{prop} 

\begin{proof}
By Proposition~\ref{P:Res} $iii)$ we have  the relation $c(S)\leq 2g(S)$ and so the $S$-decomposition number of~$c(S)$ occurs in $\delta_S$.
Proposition~\ref{P:delta} implies 
\[
c(S)=1+\max\{i\in\{0,...,3G\},\ d_S(i)=0\}.
\]
As all elements of $\NN\setminus S$ are smaller than $c(S)$, their $S$-decomposition numbers are in $\delta_S$ and we obtain
\[
g(S)=\card\{i\in\{0,...,3G\},\ d_S(i)=0\}.
\]
By Proposition~\ref{P:Res} $ii)$, the relation $m(S)\leq g(S)+1$ holds. 
This implies that the $S$-decomposition number of $m(S)$ appears in $\delta_S$ : 
\[
m(S)=\min\{i\in\{0,...,3G\},\ d_S(i)>0\}.
\]
By Proposition~\ref{P:Res},  all irreducible elements are smaller than $c(S)+m(S)-1$, which is itself smaller than~$3G$.
Hence, Proposition~\ref{P:Dandd} gives
\[
\Irr(S)=\{i \in\{0,...,3G\},\ d_S(i)=1\}.\qedhere
\]
\end{proof}

The previous representation of numerical semigroup (in terms of the vector $\delta_S$)  is similar but a little different from this used in \cite{Maria} and by people concerned with coding theory.

Even though it is quite simple, the computation of $c(S),m(S)$ and $g(S)$ from $\delta_S$ has a non negligible cost.
We represent a numerical semigroup $S$ of genus $g\leq G$ by $(c(S),g(S),c(S),\delta_S)$. 
In an algorithmic context, if the variable \texttt{S} stands for a numerical semigroups  we use:

-- \tt{S.c}, \tt{S.g} and \tt{S.m} for the integers $c(S)$, $g(S)$ and $m(S)$;

-- \tt{S.d[i]} for the integer $d_S(i)$.

For example the following Algorithm initializes a representation of the semigroup $\NN$ ready for an exploration of the tree $\mathcal{T}_G$ (the tree of numerical semigroup of genus at most $G$.)

\begin{algorithm}
\caption{\small Returns the root of $\mathcal{T}_G$}\label{A:Root}
\begin{algorithmic}
\small
\Function{Root}{\tt{G}}
\State \tt{R.c} $\gets$ 0 \Comment{\tt{R} stands for $\NN$}
\State \tt{R.g} $\gets$ 0
\State \tt{R.m} $\gets$ 1 
\For{\tt{x} from $0$ to $3\,\tt{G}$}
	\State\tt{R.d[x]}$\gets1+\left\lfloor\frac{\tt{x}}2\right\rfloor$
\EndFor
\State{\textbf{return} \tt{R}}
\EndFunction
\end{algorithmic}
\end{algorithm}

We can now describe an algorithm that returns the representation of the semigroup $S^x$ from that of the semigroup $S$ where $x$ is an irreducible element of $S$ greater than $c(S)$.

\begin{algorithm}
\caption{\small Returns the son $S^x$ of $S$ with $x\in\Irr(S)\cap[c(S),c(S)+m(S)[$.}\label{A:Son}
\begin{algorithmic}[1]
\small
\Function{Son}{\tt{S},\tt{x},\tt{G}}
\State \tt{$\tt{S}^\tt{x}$.c} $\gets$ $\tt{x}+1$
\State \tt{$\tt{S}^\tt{x}$.g} $\gets$ $\tt{S.g}+1$
\If{$\tt{x}>\tt{S.m}$}
	\State \tt{$\tt{S}^\tt{x}$.m} $\gets$ \tt{S.m}
\Else
	\State \tt{$\tt{S}^\tt{x}$.m} $\gets$ $\tt{S.m}+1$
\EndIf
\State\tt{$\tt{S}^\tt{x}$.d} $\gets$ \tt{S.d}\Comment{copy all the decomposition numbers}
\For{\tt{y} from \tt{x} to $3\,\tt{G}$}
	\If{$\tt{S.d}[\tt{y}-\tt{x}]>0$}
		\State $\tt{$\tt{S}^\tt{x}$.d}[\tt{y}]\gets\tt{$\tt{S}$.d}[\tt{y}]-1$ \Comment{decrease the decomposition number by $1$}
	\EndIf
\EndFor
\State{\textbf{return} $\tt{S}^\tt{x}$}
\EndFunction
\end{algorithmic}
\end{algorithm}

\begin{prop}
\label{P:Son}
Running on $(S,x,G)$ with $g(S)\leq G$, $x\in \Irr(S)$ and~$x\geq c(S)$, Algorithm~\ref{A:Son} returns the semigroup $S^x$ in time $O(\log(G)\times G)$.
\end{prop}

\begin{proof}
Let us check the correctness of the algorithm.
By construction $S^x$ is the semigroup~$S\setminus\{x\}$.
Thus the genus of $S^x$ is $g(S)+1$, see Line~\tt{3}.
Every integer of $I=[x+1,+\infty[$ lies in $S$ since $x$ is greater than~$c(S)$, so the interval $I$ is included in $S^x$.
As $x$ does not belong to $S^x$, the conductor of $S^x$ is $x+1$, see Line~\tt{2}.
For the multiplicity of $S^x$ we have two cases. 
First, if $x>m(S)$ holds  then $m(S)$ is also in $S^x$ and so $m(S^x)$ is equal to $m(S)$.
Assume now $x=m(S)$. 
The relation $x(S)\geq c(S)$ and the characterization of~$m(S)$ implies $x=m(S)=c(S)$.
Thus $S^x$ contains $m(S)+1$ which is $m(S^x)$.
The initialization of~$m(S^x)$ is done by Lines \tt{4} to \tt{8}. 
The correctness of the computation of $\delta_{S^x}$ (see Proposition~\ref{P:delta} for the definition of $\delta_{S^x}$) done from Line \tt{9} to Line \tt{15} is a direct consequence of~Proposition~\ref{P:Dandd}. 

Let us now prove the complexity statement. 
Since by relations $ii)$ and $iii)$ of Proposition~\ref{P:Res} we have $x\leq 3G$ together with $m(S)\leq G+1$, each line from \tt{2} to \tt{8} is done in time $O(\log(G))$.
The \textbf{for} loop needs $O(G)$ steps and each step is done in time $O(\log(G))$. 
Summarizing, these results give that the algorithm runs in time $O(\log(G)\times G)$.
\end{proof}

\begin{algorithm}[hb!]
\caption{\small Returns an array containing the value of $n_g$ for $g\leq G$}\label{A:Count}
\begin{algorithmic}[1]
\small
\Function{Count}{\tt{G}}
\State \tt{n} $\gets [0,...,0]$ \Comment $\tt{n[g]}$ stands for $n_g$ and is initialized to $0$
\State Stack \tt{stack} \Comment{the empty stack}
\State \tt{stack.push(\textsc{Root}(\tt{G}))}
\While{\tt{stack} is not empty} 
	\State $\tt{S} \gets \tt{stack.top()}$
	\State \tt{stack.pop()}
	\State $\tt{n[S.g]}\gets\tt{n[S.g]}+1$
	\If{$\tt{S.g}< \tt{G}$}
		\For{\tt{x} from $\tt{S.c}$ to $\tt{S.c}+\tt{S.m}$}
			\If{$\tt{S.d[x]\,=\,1}$}			
			      \State 
$\tt{stack.push}(\textsc{Son}(\tt{S},\tt{x},\tt{G}))$
			\EndIf
		\EndFor
	\EndIf
\EndWhile
\State{\textbf{return} \tt{n}}
\EndFunction
\end{algorithmic}
\end{algorithm}

\begin{prop}
Running on $G\in\NN$, Algorithm~\ref{A:Count} returns the values of $n_g$ for $g\leq G$  in time
\[
O\left(\log(G)\times G \times \sum_{g=0}^G n_g\right)
\]
and its space complexity is $O(\log(G)\times G^3)$.
\end{prop}

\begin{proof}
The correctness of the algorithm is a consequence of Proposition~\ref{P:Son} and of the description of the tree~$\mathcal{T}$ of numerical semigroups.

For the time complexity, let us remark that Algorithm~\textsc{Son} is called for every semigroup of the tree~$\mathcal{T}_G$ (the tree of semigroups of genus $\leq G$). 
Since there are exactly $N=\sum_{g=0}^G n_g$ such semigroups, the time complexity of \textsc{Son} established in Proposition~\ref{P:Son} guarantees that the running time of \textsc{Count} is in~$O(\log(G)\times G\times N)$, as stated.

Let us now prove the space complexity statement.
For  this we need to describe the stack through  the run of the algorithm.
Since the stack is filled  with a depth first search algorithm, it has two properties.
The first one is that reading the stack from the bottom to the top, the genus increases.  
The second one is that, for all $g\in[0,G]$, every semigroup of genus~$g$ in the stack has the same father.
As the number of sons of a  semigroup $S$ is the number of $S$-irreducible elements in the set $\{c(S),...,c(S)+m(S)-1\}$,  a semigroup $S$ has at most $m(S)$ sons. 
By~Proposition~\ref{P:Res} $ii)$, this implies that a semigroup of genus $g$ has at most $g+1$ sons.
Therefore the stack contains at most $g+1$ semigroups of genus $g+1$ for $g\leq G$. 
So the size of the stack is bounded~by 
\[
M=\sum_{g=0}^Gg=\frac{G(G+1)}2.
\]
A semigroup $S$ is represented by a quadruple $(c(S),g(S),m(S),\delta_S)$. 
By relations $ii)$ and $iii)$ of Proposition~\ref{P:Res}, we have $c\leq 2g(S)$ and $m\leq g(S)+1$. 
As $g(S)\leq G$ holds, the integers~$c$, $g$  and $m$ of the representation of $S$ require a memory space in $O(\log(G))$. 
The size of~$\delta_S=(d_S(0),...,d_S(3G))$ is exactly~$3G+1$. 
Each entry of $\delta_S$ is the $S$-decomposition number of an integer smaller than $3G$ and hence
requires $O(\log(G))$ bytes of memory space.
Therefore the space complexity of $\delta_S$ is in~$O(\log(G)\times G)$, which implies that the space complexity of the \textsc{Count} algorithm is
\[
O(\log(G)\times G\times M)= O(\log(G)\times G^3).\qedhere
\]
\end{proof}

\section{Technical optimizations and results}

\label{S:Opti}

Even though there are asymptotically faster algorithms than the one presented here,
thank to careful optimizations, we were able to compute $n_g$ for much larger
genuses than before. This is due to the fact that our algorithm is particularly
well suited for the current processor architecture. In particular, it allows
to use parallelism at various scales (parallel branch exploration,
vectorization)...

To get the greatest speed from modern processors, we used several optimization
tricks, which we will elaborate in the following section:

\begin{itemize}

\item Vectorization (\MMX, \SSE instructions sets) and careful memory alignment; 
\item Shared memory multi-core computing using \CilkP for low level
  enumerating tree branching;
\item Partially derecursived algorithm using a stack;
\item Avoiding all dynamic allocation during the computation: everything is
  computed ``in place'';
\item Avoiding all unnecessary copy (the cost of the \textsc{Son} algorithm is
  roughly the same as copying);
\item Aggressive loop unrolling: the main loop is unrolled by hand using some
  kind of Duff's device;
\item Careful choice of data type (\verb|uint_fast8_t| for decomposition
  number, vs \verb|uint_fast64_t| for all indexes).
\end{itemize}

The source code of our algorithm is available in \cite{code}.
\subsection{Vectorization}

Assume for example that we want to construct the tree~$\mathcal{T}_{100}$
of all numerical semigroups of genus smaller than $100$.  In this case, the
representation of numerical semigroups given in Section~\ref{S:DecNumber} uses
decomposition numbers of integers smaller than~$300$.  By Lemma~\ref{L:RepN},
such a decomposition number is smaller than $151$ and requires~$1$ byte of
memory.  Thus at each \textbf{for} step of Algorithm~\textsc{Son}, the CPU
actually works on $1$ byte. However current CPUs usually work on $8$ bytes and
even on $16$~bytes using vector extensions. 
The first optimization uses this point. 

To go further we must specify that the array of decomposition numbers in the
representation of a semigroup corresponds to consecutive bytes in memory. In
the \textbf{for} loop of Algorithm~\textsc{Son} we may imagine two cursors:
the first one, denoted \texttt{src} pointing to the memory byte of
$\texttt{S.d[0]}$ and the second one, denoted \texttt{dst} pointing to the
memory byte $\texttt{T.d[y]}$.  Using these two cursors, Lines \tt{10} to
\tt{14} of Algorithm~\ref{A:Son} can be rewritten as follows:

\vspace{1em}
\small
\begin{algorithmic}
\State $\tt{src}\gets \text{address}(\tt{S.d[0]})$
\State $\tt{dst}\gets \text{address}(\tt{T.d[x]})$
\State $\tt{i}\gets 0$
\While{$\tt{i}\leq \tt{3G}-\tt{x}$}
\If{$\text{content}(\tt{src})>0$}
\State $\textbf{decrease}\ \text{content}(\tt{dst})$ \textbf{by} $1$\EndIf
\State \textbf{increase} \tt{src},\tt{dst},\tt{i} \textbf{by} $1$
\EndWhile
\end{algorithmic}
\vspace{1em}

\noindent In this version we can see that the cursors \texttt{src} and
\texttt{dst} move at the same time and that the modification of the value
pointed by \texttt{dst} only needs to access the values pointed by \tt{src}
and \tt{dst}.  We can therefore work in multiple entries at the same time
without collision.  Current CPUs allow this thanks to the \SIMD technologies as
\MMX, \SSE, etc.  The acronym \SIMD~\cite{WikipediaSIMD} stands for Single Operation Multiple Data.  
We used \SSEV~\cite{WikipediaSSE, IntelSSE} technology as it allows for the largest speedup\footnote{A much greater  speedup can be certainly obtained using \texttt{AVX2} technology \cite{AVX2}. However, at this time, we cannot access a performant computer with this set of instructions.}.  
This need to respect some constraints in the memory
organization of the data, namely ``memory alignment''. Recall that an address
is 16~bytes aligned if it is a multiple of 16. \SSE memory access are much
faster for aligned memory.

The computation of the children is then performed as follows. First, the
parent's decomposition numbers are copied in the children's using the
following \CPP code
{
\small
\begin{alltt}
void copy_blocks(dec_blocks &dst, dec_blocks const &src) {
  for (ind_t i=0; i<NBLOCKS; i++) dst[i] = src[i];
}
\end{alltt}}

Here \tt{dec\_blocks} is a type for arrays of 16~bytes blocks whose size
\tt{NBLOCKS} are just large enough to store the decomposition numbers (that is
$3G$ rounded up to a multiple of 16). The instruction \tt{dst[i] = src[i]}
actually copies a full 16~bytes block.

Then the core of the \textbf{while} loop in the preceding algorithm is
translated as a \texttt{for} loop as follows (recall that $x$ denotes the generator of
the father which is to be removed in the children):

{\small
\begin{alltt}
start = x >> 4;      // index of the block containing x
shift = x & 0xF;     // offset of x inside the block
...                  // some specific instructions to handle the first incomplete block.
for (long int i=start+1; i<NBLOCKS; i++) {
  block = load_unaligned_epi8(src + ((i-start)<<4) - shift);
  dst[i] -= ((block != zero) & one);
}
\end{alltt}
}

The instruction \tt{load\_unaligned\_epi8} (specific to \SSE technology) loads
 16 consecutive entries of the decomposition number
of the parent (called \tt{src} semigroup) in the variable \tt{block}.
Those entries will be used to compute
the entries $16i, \dots 16i+15$ of the children semigroups. Since the removed
generator~$x$ is not necessarily a multiple of $16$, the data are not aligned
in memory, hence the use of a specific instruction. The \tt{zero}
(resp. \tt{one}) constants are initialized as 16 bytes equal to $0$
(resp. $1$). The comparison \tt{(block != zero)} therefore returns a block
which contains $0$ in the bytes corresponding the $0$ entries of block and
$255$ in the non zero one. This result is then bitwise and-ed with one so that
the instruction actually performs a 16~bytes parallel version of

\begin{center}
$\tt{dst} \gets \tt{dst} - \textbf{ if } \tt{block} \neq 0 \textbf{ then } 1 \textbf{ else } 0$
\end{center}

\noindent which is equivalent to Lines \tt{10} to \tt{14} of Algorithm~\ref{A:Son}.
\medskip

As we previously said, to gain more speed this core loop is actually unrolled
using some kind of Duff device~\cite{WikipediaDuff}. 

\subsection{Parallel tree exploration using \CilkP}

Our second optimization is to use parallelism on exploration of the tree.
Today, CPUs of personal computers have several cores (2, 4 or more).  The given
version of our exploration algorithm uses a single core and so a fraction only
of the power of a CPU. The idea here is that different branches of the tree
can be explored in parallel by different cores of the computer. The tricky
part is to ensure that all cores are busy, giving a new branch when a core is
done with a former one. Fortunately there is a technology called
\CilkP~\cite{CilkIntel, CilkRefman} which is particularly well suited for
those kinds of problems. For our computation, we used the free version which is
integrated in the latest version of the GNU~C compiler~\cite{GCCcilk}.

\Cilk is a general-purpose language designed for multithreaded parallel
computing. The \CPP incarnation is called \CilkP. The biggest principle behind
the design of the \Cilk language is that the programmer should be responsible
for \emph{exposing} the parallelism, identifying elements that can safely be
executed in parallel; it should then be left to the run-time environment,
particularly the scheduler, to decide during execution how to actually divide
the work between cores.

The crucial thing is that two keywords are all that are needed to start
using the parallel features of \CilkP:
\begin{itemize}
\item\tt{cilk\_spawn}: used on a procedure call, indicates that the call can
  safely operate in parallel with the remaining code of the current
  function. Note that the scheduler is not obliged to run this procedure in
  parallel; the keyword merely alerts the scheduler that it can do so.
\item\tt{cilk\_sync}: indicates that execution of the current procedure cannot
  proceed until all previously spawned procedures have completed and returned
  their results to the current frame.
\end{itemize}
As a consequence, to get a parallel version of the recursive
Algorithm~\ref{A:ExpRec}, one only needs to modify it as

\vspace{1em}
{\small
\begin{algorithm}
\caption{\small \Cilk version of Algorithm~\ref{A:ExpRec}}
\begin{algorithmic}
\Procedure{ExploreRec}{\tt{S}, \tt{G}}
	\If{$g(\tt{S})< \tt{G}$}
		\For{\tt{x} from $c(\tt{S})$ to $c(\tt{S})+m(\tt{S})$}
			\If{$\tt{x}\in \Irr(\tt{S})$}
				\State\textbf{cilk\_spawn} 
\Call{ExploreRec}{$\tt{S}^\tt{x}, \tt{G}$}
			\EndIf
		\EndFor
	\EndIf
\EndProcedure
\end{algorithmic}

\end{algorithm}
}
\vspace{1em}

We just tell \CilkP that the subtrees rooted at various children can be
explored in parallel.
Things are actually only a little bit more complicated. First we have to
gather the results of the exploration. If we simply write
\begin{center}
\tt{result}[$g(\tt{S})$] $\gets \tt{result}[g(\tt{S})] + 1$
\end{center}
then we face the problem of two cores incrementing the same variable at the
same time.
 Incrementing a variable is actually done in 3 steps: reading the
value from the memory, adding one, storing back the result. 
Since there is by
default non synchronization, the following sequence of actions for two cores
 is possible: Read1 / Read2 / Add1 / Add2 / Store1 / Store2. Then the two
cores perform the same modification resulting in incrementing the variable
only once. This is called a data race and leads to nondeterministic wrong
results. To cope with those kinds of synchronization problems, \CilkP provides the
notion of \emph{reducer} which are variables local to each thread which are
gathered (here added) when a thread is finishing its job.

A more important problem is that the cost of a recursive call in non
negligible. Using \CilkP recursive calls instead of \CPP calls makes
it even worse. The solution we use is to switch back to the non recursive
version using a stack when the genus is close to the target genus. This leads
to the following \CilkP code:
\vspace{1em}

{\small
\begin{verbatim}
void explore(const Semigroup &S) {
  unsigned long int nbr = 0;
  if (S.g < MAX_GENUS - STACK_BOUND) {
    auto it = generator_iter<CHILDREN>(S); 
    while (it.move_next()) { //iterate along the children of S
      auto child = remove_generator(S, it.get_gen()).
      cilk_spawn explore(child);
      nbr++;
    }
    cilk_results[S.g] += nbr;
  }
  else explore_stack(S, cilk_results.get_array());
}
\end{verbatim}
}
\vspace{1em}

Note that in our version, we found that the \tt{STACK\_BOUND} optimal value was
around 10 to 12 for genus in the range~$45\dots 67$ so that \tt{explore\_stack}
is used more than $99\%$ of the time. The \CilkP recursive function does
actually very little work but ensures that the work is balanced between the
different cores.

\subsection{Various technical optimizations}

Using vectorization and loop unrolling as described previously leads to an
extremely fast \textsc{Son} algorithm. Indeed, its cost is comparable to the cost
of copying a semigroup. It is therefore crucial for performance to avoid any
extra cost. We list here various places where unnecessary cost can be avoided.

\subsubsection*{Avoiding all unnecessary copy} We also used a trick to avoid
copying from and to the top of the stack. Indeed, the main loop performs the
following sequences of operations:

\vspace{1em}

{\small
\begin{algorithmic}
  \State $\tt{S} \gets \tt{stack.top()}$
  \State \tt{stack.pop()}
  \For{all children $\tt{S}^\tt{x}$ of \tt{S}}
  \State $\tt{S.push}(\tt{S}^\tt{x})$
  \EndFor
\end{algorithmic}
}

\vspace{0.5em}

If we use a stack of semigroup, we can construct $\tt{S}^\tt{x}$ directly into
the stack memory but we have to copy the top of the stack to \tt{S}. 
In \cite{Zhai}, A.~Zhai establishes that the limit of the quotient   $\frac{n_g}{n_{g-1}}$, when $g$ go to infinity, is the golden ratio $\phi\approx 1.618$. Therefore this single copy is far from
being negligible. The trick is to use a level of indirection, replacing the
stack by an array of semigroups $A$ and an array of indexes $I$ pointing to the array of
semigroups. The array $I$ can be viewed as a permutation of
the array~$A$. Now instead of copying \tt{S} out of the stack, we keep
it on the stack, pushing the children in the second position by exchanging the
indexes in $I$. Here is the relevant part of the code:

\vspace{0.5em}

{\small
\begin{verbatim}
  Semigroup data[MAX_GENUS-1], *stack[MAX_GENUS], *current;
  Semigroup **stack_pointer = stack + 1;
  for (ind_t i=1; i<MAX_GENUS; i++) stack[i] = &(data[i-1]);
  [...]
  while (stack_pointer != stack) {
      --stack_pointer;
      current = *stack_pointer;
      [...] for each children {
        *stack_pointer = *(stack_pointer + 1);
        [...] construct the children in **stack_pointer
        [...] using the parent in *current
        stack_pointer++;
        [...]
      }
      *stack_pointer = current;
  }
\end{verbatim}
}

\vspace{1em}

\subsubsection*{Avoiding dynamic allocation} Compared to the \textsc{Son}
algorithm, dynamic allocation costs orders of magnitude more. Therefore, during
the derecursived stack algorithm, we only allocate (on the system stack rather
than on the heap) the stack of semigroups. No further allocations are done.

\subsubsection*{Pointer arithmetic and indexes}

Due to the way \CPP does its pointer arithmetic, even if the index in the
array are less than $3G$ and therefore fits in $8$~bits, we use $64$~bits
indexes (namely \verb|uint_fast64_t|) to avoid conversion and sign extension
when computing addresses of indexed elements. This single standard trick save
10\% of speed.

\section{Results}

Running the \Cilk version of our optimized algorithm we have explored the tree
of numerical semigroups up to genus $67$. The computations were done on a shared
$64$ core AMD Opteron\texttrademark{} Processor 6276. As other heavy
calculations were running on the machine, we only used $32$ cores. The
computations took $18$~days. The values of~$n_g$ for $g\leq 67$ are:

\begin{center}
\begin{tabular}{|r|r||r|r||r|r|}
\hline
g & $n_g$ & g & $n_g$ & g & $n_g$ \\
\hline
0 & 1 & 23 & 170\,963 & 46 & 14\,463\,633\,648\\
1 & 1 & 24 & 282\,828 & 47 & 23\,527\,845\,502\\
2 & 2 & 25 & 467\,224 & 48 & 38\,260\,496\,374\\
3 & 4 & 26 & 770\,832 & 49 & 62\,200\,036\,752\\
4 & 7 & 27 & 1\,270\,267 & 50 & 101\,090\,300\,128\\
5 & 12 & 28 & 2\,091\,030 & 51 & 164\,253\,200\,784\\
6 & 23 & 29 & 3\,437\,839 & 52 & 266\,815\,155\,103\\
7 & 39 & 30 & 5\,646\,773 & 53 & 433\,317\,458\,741\\
8 & 67 & 31 & 9\,266\,788 & 54 & 703\,569\,992\,121\\
9 & 118 & 32 & 15\,195\,070 & 55 & 1\,142\,140\,736\,859\\
10 & 204 & 33 & 24\,896\,206 & 56 & 1\,853\,737\,832\,107\\
11 & 343 & 34 & 40\,761\,087 & 57 & 3\,008\,140\,981\,820\\
12 & 592 & 35 & 66\,687\,201 & 58 & 4\,880\,606\,790\,010\\
13 & 1\,001 & 36 & 109\,032\,500 & 59 & 7\,917\,344\,087\,695\\
14 & 1\,693 & 37 & 178\,158\,289 & 60 & 12\,841\,603\,251\,351\\
15 & 2\,857 & 38 & 290\,939\,807 & 61 & 20\,825\,558\,002\,053\\
16 & 4\,806 & 39 & 474\,851\,445 & 62 & 33\,768\,763\,536\,686\\
17 & 8\,045 & 40 & 774\,614\,284 & 63 & 54\,749\,244\,915\,730\\
18 & 13\,467 & 41 & 1\,262\,992\,840 & 64 & 88\,754\,191\,073\,328\\
19 & 22\,464 & 42 & 2\,058\,356\,522 & 65 & 143\,863\,484\,925\,550\\
20 & 37\,396 & 43 & 3\,353\,191\,846 & 66 & 233\,166\,577\,125\,714\\
21 & 62\,194 & 44 & 5\,460\,401\,576 & 67 & 377\,866\,907\,506\,273\\
22 & 103\,246 & 45 & 8\,888\,486\,816 & &\\
\hline
\end{tabular}
\end{center}

As the reader can check the convergence of sequence $\frac{n_g}{n_{g-1}}$ established by A.~Zhai in \cite{Zhai} is very slow: $\frac{n_{67}}{n_{66}}\approx 1.62$.

\subsection{Wilf's conjecture}

In the paper \cite{Wilf} of 1978, H.S. Wilf conjectured that all numerical semigroup $S$ satisfy the relation 
\[
\card(\Irr(S))\geq \frac{c(S)}{c(S)-g(S)}.
\]
Since the work of M.~Bras-Amor{\'o}s, see \cite{BrasAmoros2008}, we yet know that all numerical semigroups of genus $g\leq 50$ satisfy Wilf's conjecture. With our exploration algorithm we have proved that there is no counterexample to Wilf's conjecture up to genus $60$.
We have tested the Wilf's conjecture on a different machine than the one used to determine $n_{67}$.
As its performance is lower we have only tested the conjecture for $g\leq 60$. 

\subsection{Timings}

In this section we summarize the timing improvements through the different 
optimizations of our algorithm.

The following table shows the time needed by the algorithm for computing 
the values of $n_g$ 
for $g\leq G$ with $30\leq G \leq 40$
on a machine equipped with  an \tt{Intel\texttrademark{} i5-3570K} CPU running at \texttt{3.4GHz}
and \texttt{8GB} of memory.
All algorithms are executed on only one core. Algorithm \texttt{breadth} is 
based on a breadth exploration of the tree while Algorithm \texttt{depth} 
use a depth exploration. 
These two algorithms are based on the same naive representation of numerical semigroups.
The only difference concerns the tree exploration algorithm used.
Algorithm \texttt{depth+$\delta$} is a refinement of 
\texttt{depth} based on the $S$-decomposition function. Algorithm 
\texttt{$\delta$+sse} in an optimization of \texttt{depth+$\delta$} using the 
\SIMD extension \SSE. Times are in seconds.

\vspace{1em}

\begin{center}
\begin{tabular}{|c|c|c|c|c|c|c|c|c|c|c|c|}
\hline
\texttt{Algorithm}					&  30 &  31 &  32 &  33 &  34 &   35 &  36 &  37 &  38 &  39 &  40 \\
\hline
\texttt{breadth}		& 5.0 & 8.3 &  14 &  23 &  38 & 1251 &     &     &     &     &     \\
\texttt{depth}			& 3.4 & 5.8 & 9.2 &  16 &  27 &   45 &  75 & 125 & 204 & 346 & 557 \\
\texttt{depth+$\delta$}	& 0.3 & 0.6 & 1.0 & 1.7 & 2.7 &  4.2 & 7.4 &  12 &  20 &  32 & 74  \\
\texttt{$\delta$+sse}	& 0.1 & 0.2 & 0.3 & 0.4 & 0.8 &  1.2 & 2.0 & 3.1 & 5.1 & 9.0 & 14 \\
\hline
\end{tabular}
\end{center}

\vspace{1em}

The computation of $n_g$ for $g\leq 35$ with algorithm \texttt{breadth} is very 
long because all the \texttt{8GB} of memory are consumed and the system 
must use swap memory to finish the computation. This algorithm was not launched 
for genus~$g\geq 36$.

The following table illustrates the impact of parallelization on the same 
machine based on the \tt{Intel\texttrademark{} i5-3570K} CPU which have $4$ physical cores that are able to run four threads.
  
  \begin{figure}[h!]
\begin{tabular}{|c|c|c|c|c|c|} 
\hline
Threads & 30 & 35 & 40 & 45 & 50 \\
\hline
1 &  0.11 & 1.26 & 14.9 & 182 & 2201 \\ 
2 &  0.06 & 0.65 & 7.50 &  92 & 1110\\
3 &  0.05 & 0.44 & 5.14 &  63 & 747 \\
4 &  0.04 & 0.34 & 4.02 &  48 & 489 \\
\hline
\end{tabular}
\end{figure}

The time of the one threaded algorithm  must be compared with the 
\texttt{$\delta$+sse} version of the previous table : it illustrates the 
additional cost induced by the use of \Cilk technology. 

It should be noticed that the \texttt{TurboBoost} technology \cite{WikipediaTB} 
is present on the CPU. Therefore the clock of the CPU is a bit higher when the 
number of threads 
is smaller. Therefore the gain of using more threads is a bit over than the one 
suggestested by the table.

Finally we tested our algorithm on a machine holding two 
\tt{Intel\texttrademark{} Xeon\texttrademark{} X5650} CPU running 
at \texttt{2.67GHz}.
Each of those CPU has 6 physical cores that are able to run 12 threads  thanks to the \texttt{Hyper-Threading} 
technology~\cite{WikipediaHT}: the 
machine has also 12 physical cores and  is able to run 24 threads.
However, when more 
than $12$ threads are running, the computation engines are shared between two 
threads and the speedup should be much less when
adding more cores. 
The following table resumes the time needed by the algorithm 
to explore $\mathcal{T}_{50}$ on different numbers of threads (the double 
vertical line delimite the use of \texttt{Hyper Threading}). For reference,
we put in the column called \texttt{C++} the computation time of the same
program compiled without \Cilk.
\begin{figure}[h!]
\begin{tabular}{|c|c|c|c|c|c|c||c|c|c|}
\hline
 Threads & \texttt{C++} & 1 & 2 & 4 & 8 & 12 & 16 & 20 & 24\\
 \hline
 Time (s) & 3588 & 3709 & 1865 & 932.4 & 486.8 & 325.7 & 311.2 & 302.3 & 290.2\\
 \hline
\end{tabular}
\end{figure}

Further improvements on the computation of $n_g$  or on the verification of Wilf's conjecture will be published on our home pages and on \cite{code}.
\vspace{2em}

\textbf{Acknowledgment.} We would like to thank Shalom Eliahou for introducing
us to the tree of numerical semigroups and for interesting discussions on the
subject. We are grateful to Laura Grigori who suggested us that Cilk could be
a good technology to tackle the tree-level parallelism for this kind of
problems.  We would finally like to thanks Nathan Cohen who suggested the two
levels of indirection stack optimization.

\bibliographystyle{plain} 
\bibliography{biblio}
\newpage

\noindent \textbf{Jean Fromentin}

ULCO, LMPA J.~Liouville, B.P. 699, F-62228 Calais, France

CNRS, FR 2956, France

\url{fromentin@lmpa.univ-littoral.fr}

\vspace{2em}

\noindent \textbf{Florent Hivert}

Bureau 33, Laboratoire de Recherche en Informatique (UMR CNRS 8623)

B\^atiment 650, Universit\'e Paris Sud 11, 91405 Orsay CEDEX, France

\url{florent.hivert@lri.fr}
\end{document}